\documentclass[10pt,a4paper,oneside,reqno]{amsart}

\usepackage{setspace}
\usepackage[totalwidth=14cm,totalheight=21.5cm]{geometry}
\usepackage[T1]{fontenc}
\usepackage[dvips]{graphicx}
\usepackage{epsfig}
\usepackage{amsmath,amssymb,amscd,amsthm}
\usepackage{subfigure}
\usepackage[latin1]{inputenc}
\usepackage{latexsym}
\usepackage{graphics}
\usepackage{colortbl} 
\usepackage{color}
\usepackage{ae}
\usepackage{enumitem}
\usepackage[all]{xy}
\usepackage{scalerel}
\usepackage{tikz}
\usepackage{cancel}
\usepackage{bbm,amsbsy}
\usepackage{mathrsfs}  
\usepackage[colorlinks=true,pagebackref=true]{hyperref}
\usepackage[normalem]{ulem}
\usepackage{nicefrac}
\usepackage{xfrac}
\usepackage{faktor}
\usepackage{tikz-cd} 

\usepackage{nomencl}
\usepackage[normalem]{ulem}

\makenomenclature

\makeatletter
\newtheorem*{rep@theorem}{\rep@title}
\newcommand{\newreptheorem}[2]{%
\newenvironment{rep#1}[1]{%
 \def\rep@title{#2 \ref{##1}}%
 \begin{rep@theorem}}%
 {\end{rep@theorem}}}
\makeatother

\makeatletter
\newtheorem*{rep@cor}{\rep@title}
\newcommand{\newrepcor}[2]{%
\newenvironment{rep#1}[1]{%
 \def\rep@title{#2 \ref{##1}}%
 \begin{rep@cor}}%
 {\end{rep@cor}}}
\makeatother

\makeatletter
\newtheorem*{rep@prop}{\rep@title}
\newcommand{\newrepprop}[2]{%
\newenvironment{rep#1}[1]{%
 \def\rep@title{#2 \ref{##1}}%
 \begin{rep@prop}}%
 {\end{rep@prop}}}
\makeatother

\newtheorem{cor}{Corollary}[section]

\newtheorem{thm}[cor]{Theorem}

\newtheorem{prop}[cor]{Proposition}

\newrepcor{cor}{Corollary}
\newreptheorem{thm}{Theorem}
\newrepprop{prop}{Proposition}

\newtheorem{lemma}[cor]{Lemma}

\theoremstyle{definition}
\newtheorem{dfn}[cor]{Definition}
\theoremstyle{remark}
\newtheorem{remark}[cor]{Remark}
\newtheorem*{remark*}{Remark}

\newtheorem*{notation*}{Notation}

\newlist{steps}{enumerate}{1}
\setlist[steps, 1]{itemsep=8pt,leftmargin=0cm,itemindent=.5cm,labelwidth=\itemindent,labelsep=0cm,align=left,label = \textbf{\emph{Step \arabic*}:\,}}

\newcommand{\R}{{\mathbb R}}

\newcommand{\Hyp}{\mathbb{H}}

\newcommand{\SO}{\mathrm{SO}}

\begin{document}

\setcounter{secnumdepth}{2}
\setcounter{tocdepth}{2}

\title[A remark on one-harmonic maps to the hyperbolic plane]{A remark on one-harmonic maps from a Hadamard surface of pinched negative curvature to the hyperbolic plane}

\author[Fran\c{c}ois Fillastre]{Fran\c{c}ois Fillastre}
\address{Fran\c{c}ois Fillastre: CY Cergy Paris Universit\'e
Laboratoire AGM, UMR 8088 du CNRS, F-95000 Cergy, France}\email{francois.fillastre@cyu.fr}\address{
IMAG, Universit\'e de Montpellier, CNRS, Montpellier, France
}
 \email{francois.fillastre@umontpellier.fr}

\author[Andrea Seppi]{Andrea Seppi}
\address{Andrea Seppi: CNRS and Universit\'e Grenoble Alpes, 100 Rue des Math\'ematiques, 38610 Gi\`eres, France.} \email{andrea.seppi@univ-grenoble-alpes.fr}


\thanks{The second author is member of the national research group GNSAGA.}

\maketitle

\begin{abstract}
We show that every one-harmonic map, in the sense of Trapani and Valli, from a Hadamard surface of pinched negative curvature to $\Hyp^2$ has image the interior of the convex hull of a subset of $\partial_\infty\Hyp^2$. The proof relies on Minkowski geometry, by interpreting one-harmonic maps as the Gauss maps of convex surfaces.
\end{abstract}

\vspace{0.5cm}
\section{Introduction}

Energy-minimizing maps have played an important role in Teichm\"uller theory, and more generally in the study of negatively curved Riemannian surfaces. The classical $L^2$ energy between closed hyperbolic surfaces has been largely studied, see for instance \cite{zbMATH03295470,zbMATH03608406,zbMATH04069867,zbMATH00034325,zbMATH00011240,zbMATH00120199}, leading to important new descriptions of Teichm\"uller space. In the setting of universal Teichm\"uller space, the Schoen conjecture, formulated in \cite{zbMATH00421703} and proved independently in \cite{zbMATH06699459} and \cite{zbMATH06731860}, states that every quasisymmetric homeomorphism of $\partial_\infty\Hyp^2$ admits a quasiconformal harmonic extension to $\Hyp^2$. 
The result has been generalized in \cite{benhulII} for Hadamard manifolds, which we recall are complete simply connected Riemannian manifold with everywhere non-positive sectional curvature, under the assumption of pinched negative curvature (i.e. the curvature is bounded above and below by negative constants).
 Harmonic maps with image in subsets of $\Hyp^2$ have also been studied, see for instance \cite{zbMATH00856913,zbMATH05001913,zbMATH07024058}.

In \cite{zbMATH00870993}, Trapani and Valli introduced the notion of one-harmonic maps between closed surfaces of negative curvature, defined as the critical points of a holomorphic energy functional. See also \cite{graham}. We give Definition \ref{defi1har} in the more general non-compact setting. 

\begin{dfn}\label{defi1har}
Let $(\Sigma,g)$ and $(\Sigma',h)$ be oriented Riemannian surfaces, and let $F:\Sigma\to\Sigma'$ be an orientation-preserving local diffeomorphism. Then $F$ is a \emph{one-harmonic map} if for every open set $\Omega\subset \Sigma$ with compact closure on which $F$ is a diffeomorphism onto its image, $F$ is a critical point of the functional
$$F\mapsto \int_\Omega\|\partial F\|\mathrm{dArea}_g$$
among diffeomorphisms from $\Omega$ to $F(\Omega)$ that coincide with $F$ in the complement of a compact subset of $\Omega$. 
\end{dfn}

In Definition \ref{defi1har}, $\partial F$ denotes the $(1,0)$-part of the differential of $F$, computed with respect to the complex structures underlying $g$ and $h$, and $\|\partial F\|$ is its norm, computed with respect to $g$ and $h$. 

Trapani and Valli proved the existence and uniqueness of a one-harmonic map between two closed diffeomorphic Riemannian surfaces $(\Sigma,g)$ and $(\Sigma,h)$ of negative curvature, in a prescribed isotopy class.  When $g$ and $h$ are hyperbolic metrics, the one-harmonic maps are minimal Lagrangian, meaning that their graph is a minimal Lagrangian submanifold in the product. The existence of a minimal Lagrangian diffeomorphism between closed hyperbolic surfaces in a given isotopy class is due to Labourie \cite{labourieCP} and Schoen \cite{zbMATH00421703} independently.  This has been later generalized to the context of universal Teichm\"uller theory by Bonsante and Schlenker, showing that any quasisymmetric homeomorphism of $\partial_\infty\Hyp^2$ admits a unique quasiconformal minimal Lagrangian extension to $\Hyp^2$. See also \cite{zbMATH06902491,zbMATH07060418,zbMATH07127319,survey} for related results. 

In \cite[Corollary G]{zbMATH07087153} it was proved that the image of any minimal Lagrangian map from $\Hyp^2$ to itself is a straight convex domain, as in the following definition (see also Figure \ref{fig}).

\begin{dfn}
A \emph{straight convex domain} in $\Hyp^2$ is the interior of the (hyperbolic) convex hull of a subset of $\partial_\infty\Hyp^2$ containing at least three points.  
\end{dfn}

\begin{figure}[htb]
\centering
\includegraphics[height=4.5cm]{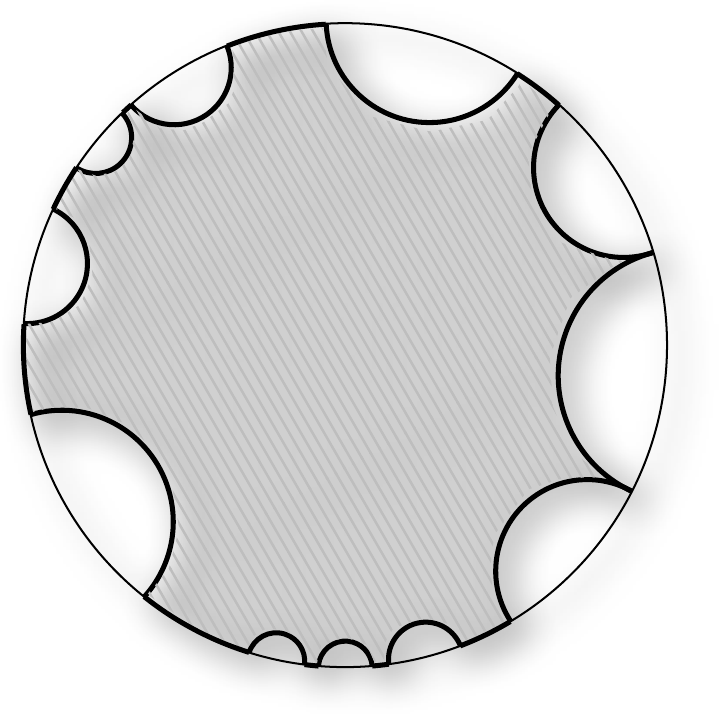}
\caption{A straight convex domain, in the Poincar\'e model of $\Hyp^2$. \label{fig}}
\end{figure}

We remark that $\Hyp^2$ itself is a straight convex domain. The purpose of this paper is to extend this result to one-harmonic maps defined on a Hadamard surface of pinched negative curvature, with target $\Hyp^2$. 

\begin{thm} \label{thm}
Let $(\Sigma,g)$ be a complete, simply connected Riemannian surface whose curvature satisfies $-c_1\leq K_g\leq -c_2$ for some constants $c_1,c_2>0$, and let $F:\Sigma\to\Hyp^2$ be a one-harmonic map. Then the image of $F$ is a straight convex domain in $\Hyp^2$ and $F$ is a diffeomorphism onto its image. 
\end{thm}

When the source has constant curvature $-1$, we recover the aforementioned result of \cite{zbMATH07087153}. The main idea of the proof comes from interpreting one-harmonic maps from a negatively curved surface to $\Hyp^2$ as the Gauss maps of (convex) surfaces in Minkowski space. The idea of realizing energy-minimizing maps as Gauss maps in Minkowski space has been largely used before, for instance for harmonic maps (which correspond to surfaces of constant mean curvature) in \cite{zbMATH03844576,zbMATH00120216} and many others, and for minimal Lagrangian maps (which correspond to surfaces of constant Gaussian curvature), see for instance \cite{bss2}. 

\subsection*{Acknowledgements}  The second author is grateful to Francesco Bonsante and Graham Smith for discussions related to Minkowski geometry and one-harmonic maps. We thank Haruko Nishi and the organizers of the JMM workshop for the opportunity to publish this work.

\section{One-harmonic maps as Gauss maps in $\R^{2,1}$}

Recall that the three-dimensional \emph{Minkowski space} is the vector space $\R^3$ endowed with the Minkowski (Lorentzian) metric:
\begin{equation}\label{mink metric}
\langle x,y\rangle=x_1y_1+x_2y_2-x_3y_3~.
\end{equation}
We will consider the hyperbolic plane in the \emph{hyperboloid model}, namely:
$$\Hyp^2=\{x\in \R^{2,1}\,|\,\langle x,x\rangle=-1,x_3>0\}~.$$

Let $\sigma:\Sigma\to \R^{2,1}$ be a smooth immersion, for $\Sigma$ a surface (without boundary). In this paper, we will always assume that $\sigma$ is \emph{spacelike}, i.e. the first fundamental form $\sigma^*\langle\cdot,\cdot\rangle$ is a Riemannian metric on $\Sigma$. An example is given by $\Hyp^2$ itself, whose first fundamental form is a complete Riemannian metric of constant curvature $-1$.

The \emph{Gauss map} of a spacelike immersion $\sigma:\Sigma\to \R^{2,1}$ is the map $G_\sigma:\Sigma\to\Hyp^2$ that maps $p\in\Sigma$ to the future unit normal vector of $\sigma$ at $p$, which is an element of $\Hyp^2$. 
The first essential step in the proof is the following statement.

\begin{lemma}\label{lemma:realize}
Let $(\Sigma,g)$ be a simply connected Riemannian surface of negative curvature and let $F:\Sigma\to\Hyp^2$ be a one-harmonic map. Then there exists a spacelike immersion $\sigma:\Sigma\to \R^{2,1}$ such that:
\begin{itemize}
\item The first fundamental form of $\sigma$ is $g$.
\item The Gauss map of $\sigma$ is $F$.
\end{itemize}
\end{lemma}

The proof of Lemma \ref{lemma:realize} relies on the following characterization of one-harmonic maps. 
Recall that, given a connection $\nabla$ on a surface $\Sigma$ and a $(1,1)$-tensor $B$ on $\Sigma$, the \emph{exterior derivative} of $B$ is defined as:
$$d^\nabla \!B(v,w)=\nabla_v(B(w))-\nabla_w(B(v))-B[v,w]~.$$

\begin{prop}[{\cite[Corollary 3.3]{graham}}]\label{lemma:graham}
Let $(\Sigma,g)$ and $(\Sigma',h)$ be oriented Riemannian surfaces. An orientation-preserving local diffeomorphism $F:\Sigma\to\Sigma'$ is one-harmonic if and only if the unique $g$-self-adjoint, positive definite $(1,1)$-tensor $B$ such that $F^*h=g(B\cdot,B\cdot)$ satisfies the Codazzi equation 
$$d^{\nabla^g}\!B=0~.$$
\end{prop}

Before proving Lemma \ref{lemma:realize}, let us make an additional remark.

\begin{remark}
The converse of Lemma \ref{lemma:realize} holds true, namely the following statement: given any spacelike immersion $\sigma:\Sigma\to \R^{2,1}$, for $\Sigma$ a simply connected surface, such that the first fundamental form $g$ of $\sigma$ is negatively curved, then   the Gauss map $G_\sigma:\Sigma\to\Hyp^2$ of $\sigma$ is one-harmonic with respect to $g$. 

Indeed, let $B$ denote the shape operator of $\sigma$, which is a $(1,1)$ symmetric tensor on $\Sigma$ defined by $B=dG_\sigma$. By the Gauss equation, $\det B=-K_g>0$, hence $B$ is either positive or negative definite. Up to changing the sign of $B$, which will not affect what follows, we can assume that $B$ is positive definite. The pull-back of the metric $h$ of $\Hyp^2$ is:
\begin{equation}\label{eq:pullback}
G_\sigma^*h=h(dG_\sigma\cdot,dG_\sigma\cdot)=g(B\cdot,B\cdot)~,
\end{equation}
where we implicitly identified $d\sigma(T_p\Sigma)$ and $dG_\sigma(T_p\Sigma)$ as linear subspaces of $\R^{2,1}$, and used that both $g$ and $h$ are the restriction of $\langle\cdot,\cdot\rangle$ to this subspace. 
Finally, by the Codazzi equation, $d^{\nabla^g}\!B=0$.
Hence $B$ satisfies the conditions of Proposition  \ref{lemma:graham}, which implies that $G_\sigma$ is one-harmonic. 
\end{remark}

Having established this remark, we shall now prove Lemma \ref{lemma:realize}.

\begin{proof}[Proof of Lemma \ref{lemma:realize}]
Let $F:\Sigma\to\Hyp^2$ be a one-harmonic map and let $B$ be the positive definite, $g$-self-adjoint $(1,1)$ tensor provided by Proposition \ref{lemma:graham}. The pair $(g,B)$ satisfies the Codazzi equation $d^{\nabla^g}\!B=0$; we claim that it also satisfies the Gauss equation in $\R^{2,1}$, namely $K_g=-\det B$. By \cite[Proposition 3.12]{zbMATH05200423}, the curvature of the metric $g(B\cdot,B\cdot)$ is
$$K_{g(B\cdot,B\cdot)}=\frac{K_g}{\det B}~,$$
since $B$ is $g$-Codazzi and invertible. But $g(B\cdot,B\cdot)=F^*h$, hence its curvature is identically $-1$. Therefore $K_g=-\det B$ as claimed. 

We have showed that the pair $(g,B)$ satisfies the Gauss-Codazzi equations in $\R^{2,1}$. Since $\Sigma$ is simply-connected, by the fundamental theorem of surfaces, there exists an immersion $\sigma:\Sigma\to\R^{2,1}$ whose first fundamental form is $g$ and whose shape operator is $B$. The Gauss map of $\sigma$ does not necessarily coincide with $F$, but it does up to post-composing $\sigma$ with an isometry of $\R^{2,1}$. We now formalize this, which is the last step of the proof. 

Observe that both $F$ and $G_\sigma$ are orientation-preserving local diffeomorphisms to $\Hyp^2$: $F$ by hypothesis, while $G_\sigma$ because its differential is identified with $B$ and satisfies $\det B=-K_g>0$. Moreover by Equation \eqref{eq:pullback}, if $h$ denotes the hyperbolic metric of $\Hyp^2$, we have
\begin{equation} \label{eq:pullbacks} 
G_\sigma^*h=g(B\cdot,B\cdot)=F^*h~.
\end{equation}
We claim that there exists an orientation-preserving isometry of $\Hyp^2$, say $A\in \SO_0(2,1)$, such that $F=A\circ G_\sigma$. This will permit us to conclude, by defining $\sigma'=A\circ \sigma$ (where we now think at $A$ as acting on $\R^{2,1}$) and observing that the first fundamental form of $\sigma'$ equals that of $\sigma$, namely $g$, and that the Gauss map of $\sigma'$ is $G_{\sigma'}=A\circ G_\sigma=F$.

The argument to prove the claim is quite standard. Since the pull-back metrics $G_\sigma^*h$ and $F^*h$ coincide, for every point $p\in\Sigma$ there exists a unique orientation-preserving linear isometry $A=A(p):T_{G_\sigma(p)}\Hyp^2\to T_{F(p)}\Hyp^2$ such that $dF_p=A\circ (dG_\sigma)_p$. Since every linear isometry between the tangent spaces at two points of $\Hyp^2$ uniquely extends to the whole $\Hyp^2$ (in fact, to $\R^{2,1}$), this defines a map from $\Sigma$ to $\SO_0(2,1)$.   But this map is locally constant, because every $p\in\Sigma$ has a neighbourhood $U_p$ on which both $F$ and $G_\sigma$ are diffeomorphisms onto their images, hence from \eqref{eq:pullbacks} we see that $F|_{U_p}=A(p)\circ G_\sigma|_{U_p}$. Since $\Sigma$ is connected, the map is constant, which means that $A\in \SO_0(2,1)$ does not depend on the point $p$.  \end{proof}

\section{Conclusion of the proof}

We are now ready to conclude the proof of Theorem \ref{thm}. Recall that, by hypothesis, the metric $g$ on $\Sigma$ is complete and $F:\Sigma\to\Hyp^2$ is a one-harmonic map; by Lemma \ref{lemma:realize} we realize $g$ as the first fundamental form of an immersion $\sigma$ whose Gauss map is $F$. First, we will apply the following statement (which actually holds in any dimension).

\begin{lemma}[{\cite[Lemma 3.1]{Bonsante}}]\label{lemma entire}
Let $\sigma:\Sigma\to\R^{2,1}$ be an immersion whose first fundamental form is a complete Riemannian metric. Then $\sigma$ is an embedding and its image is the graph of a function $f:\R^2\to\R$.
\end{lemma}

A surface in $\R^{2,1}$ which is the graph of a function $f:\R^2\to\R$ is also called \emph{entire}. Under the hypothesis of Theorem \ref{thm}, the curvature of $g$ is negative (in fact, is bounded above and below by negative constants). Since by the Gauss equation $K_g=-\det B$, this implies that $B$ is either positive definite or negative definite; up to applying a reflection is a horizontal plane, we can assume that $B$ is positive definite. This means that $im(\sigma)$ is locally strictly convex; since it is also entire, namely $im(\sigma)$ is the graph of $f:\R^2\to\R$, it follows that $f$ is a strictly convex function. Moreover $f$ is 1-Lipschitz, because $\sigma$ is spacelike. 

We are now going to apply a result from \cite{zbMATH07087153}, which we will need to interpret in terms of the Gauss map. For this purpose, we first need to recall the Legendre transformation. We provide the definition only in the special setting of our interest. Given a smooth convex function $f:\R^2\to\R$, its Legendre transformation $f^*$ is defined on $\R^2$ as
$$f^*(\mathsf y)=\sup_{\mathsf x\in\R^2}\mathsf x\cdot \mathsf y-f(\mathsf x)~,$$
where $\mathsf x\cdot \mathsf y$ denotes the standard scalar product of $\R^2$.
It is well-known that $f^*$ is a lower semicontinuous convex function with values in $\R\cup\{+\infty\}$. Observe that $f^*(\mathsf y)<+\infty$ if and only if there exists some constant $C$ such that $f(\mathsf x)>\mathsf x\cdot\mathsf y+C$ for every $\mathsf x\in\R^2$. Using the expression \eqref{mink metric} of the Minkowski metric, this is equivalent to saying that $im(\sigma)$ has a support plane orthogonal to the vector $(\mathsf y,1)$. For instance, if $\mathsf y=Df(\mathsf x)$, then $f^*(\mathsf y)<+\infty$, because from the metric \eqref{mink metric} one sees that the tangent plane to $im(\sigma)$ at the point $(\mathsf x,f(\mathsf x))$ is orthogonal to the vector $(Df(\mathsf x),1)$. 

Now, since $\sigma$ is spacelike, or equivalently $f$ is smooth and $1$-Lipschitz, it is not hard to see that the essential domain of $f^*$, namely the set 
$$\mathrm{ess}(f^*):=\{\mathsf y\,|\,f^*(\mathsf y)<+\infty\}~,$$ is contained in the closed unit disc in $\R^2$, and moreover its interior is the image $Df(\R^2)$ of the gradient mapping of $f$. The essential domain $\mathrm{ess}(f^*)$ is described in the following theorem (see also \cite[Theorem A']{xin} for a related statement).

\begin{thm}[{\cite[Theorem 4.4]{zbMATH07087153}}]\label{thm gauss image}
Let $\sigma:\Sigma\to\R^{2,1}$ be a spacelike convex entire embedding whose first fundamental form $g$ has pinched negative curvature, that is there exist constants $c_1,c_2>0$
such that $-c_1\leq K_g\leq -c_2$. Let $f:\R^2\to\R$ be the function whose graph is $im(\sigma)$. Then $\mathrm{ess}(f^*)$ is the convex hull in $\R^2$ of the set of points of $\mathbb S^1$ on which $f^*$ is finite. 
\end{thm}

From the discussion preceding Theorem \ref{thm gauss image}, the unit normal vector of the immersion $\sigma$ at a point $p$ is proportional to the vector $(Df(\mathsf x),1)$, where $\mathsf x$ is the vertical projection to $\R^2$ of $\sigma(p)$. This shows that, if $\pi:\Hyp^2\to\mathbb D^2$ denotes the radial projection from the hyperboloid to the open unit disc at height 1 in $\R^{2,1}$ (namely the Klein model of the hyperbolic plane), then $Df(\R^2)=\pi\circ G_\sigma(\Sigma)$, where $G_\sigma:\Sigma\to\Hyp^2$ is the Gauss map of $\sigma$ as usual. In conclusion, we summarize the proof as follows:

\begin{proof}[Proof of Theorem \ref{thm}]
Let $F:\Sigma\to\Hyp^2$ be a one-harmonic map from a simply connected surface and suppose that the metric $g$ on $\Sigma$ is complete and has pinched negative curvature. By Lemma \ref{lemma:realize} there exists an immersion $\sigma:\Sigma\to\R^{2,1}$ having first fundamental form $g$ and Gauss map $F$. 

The Gauss map $F$ is a local diffeomorphism because the determinant of its differential, which equals the opposite of the curvature of $g$, does not vanish. By Lemma \ref{lemma entire}, $im(\sigma)$ is entire, and (up to applying a reflection in the horizontal plane) is the graph of a strictly convex smooth function $f:\R^2\to\R$. This implies also that $F$ is injective, by a standard argument: if $\sigma$ had the same unit normal vector at two different points, by convexity its image would contain the segment connecting the two points, but this would contradict strict convexity of $f$. Hence by the invariance of domain, $F$ is a diffeomorphism onto its image. 

By the discussion in this section, the image of the Gauss map of $\sigma$ equals $Df(\R^2)$, which is the interior of $\mathrm{ess}(f^*)$. By the assumption of pinched negative curvature and Theorem \ref{thm gauss image}, $\mathrm{ess}(f^*)$ is the convex hull of a subset of $\mathbb S^1=\partial_\infty\Hyp^2$. Hence the image of $F$ is the interior of such convex hull (and the subset of $\mathbb S^1$ necessarily contains at least three points, because the Gauss map $F$ is a diffeomorphism onto its image, hence its image has non-empty interior). This concludes the proof. 
\end{proof}

\bibliographystyle{alpha}
\bibliographystyle{ieeetr}
\bibliography{biblio.bib}

\begin{thebibliography}{10}

\bibitem{zbMATH05001913}
T.~K.~K. {Au} and T.~Y.~H. {Wan}.
\newblock {Images of harmonic maps with symmetry}.
\newblock {\em {Tohoku Math. J. (2)}}, 57(3):321--333, 2005.

\bibitem{zbMATH06731860}
Y.~{Benoist} and D.~{Hulin}.
\newblock {Harmonic quasi-isometric maps between rank one symmetric spaces}.
\newblock {\em {Ann. Math. (2)}}, 185(3):895--917, 2017.

\bibitem{benhulII}
Y.~{Benoist} and D.~{Hulin}.
\newblock Harmonic quasi-isometric maps ii : negatively curved manifolds.
\newblock {\em Preprint arXiv:1702.04369}, 2017.

\bibitem{Bonsante}
F.~Bonsante.
\newblock Flat spacetimes with compact hyperbolic {C}auchy surfaces.
\newblock {\em J. Differential Geom.}, 69(3):441--521, 2005.

\bibitem{zbMATH06902491}
F.~{Bonsante} and A.~{Seppi}.
\newblock {Area-preserving diffeomorphisms of the hyperbolic plane and
  \(K\)-surfaces in anti-de Sitter space}.
\newblock {\em {J. Topol.}}, 11(2):420--468, 2018.

\bibitem{survey}
F.~Bonsante and A.~Seppi.
\newblock Anti-de {S}itter geometry and {T}eichm{\"u}ller theory.
\newblock In K.~O. V.~Alberge and A.~Papadopoulos, editors, {\em In the
  tradition of Thurston (to appear)}. Springer Verlag, 2020.

\bibitem{bss2}
F.~{Bonsante}, A.~{Seppi}, and P.~{Smillie}.
\newblock Complete cmc hypersurfaces in minkowski (n+1)-space.
\newblock {\em Preprint arXiv:1912.05477}, 2019.

\bibitem{zbMATH07087153}
F.~{Bonsante}, A.~{Seppi}, and P.~{Smillie}.
\newblock {Entire surfaces of constant curvature in Minkowski 3-space}.
\newblock {\em {Math. Ann.}}, 374(3-4):1261--1309, 2019.

\bibitem{zbMATH03295470}
C.~J. {Earle} and J.~{Eells}.
\newblock {A fibre bundle description of Teichm\"uller theory}.
\newblock {\em {J. Differ. Geom.}}, 3:19--43, 1969.

\bibitem{zbMATH00856913}
Z.-C. {Han}, L.-F. {Tam}, A.~{Treibergs}, and T.~{Wan}.
\newblock {Harmonic maps from the complex plane into surfaces with nonpositive
  curvature}.
\newblock {\em {Commun. Anal. Geom.}}, 3(1):85--114, 1995.

\bibitem{zbMATH05200423}
K.~{Krasnov} and J.-M. {Schlenker}.
\newblock {Minimal surfaces and particles in 3-manifolds.}
\newblock {\em {Geom. Dedicata}}, 126:187--254, 2007.

\bibitem{labourieCP}
F.~Labourie.
\newblock Surfaces convexes dans l'espace hyperbolique et {${\bf C}{\rm
  P}^1$}-structures.
\newblock {\em J. London Math. Soc. (2)}, 45(3):549--565, 1992.

\bibitem{zbMATH07024058}
Q.~{Li}.
\newblock {On the uniqueness of vortex equations and its geometric
  applications}.
\newblock {\em {J. Geom. Anal.}}, 29(1):105--120, 2019.

\bibitem{zbMATH06699459}
V.~{Markovic}.
\newblock {Harmonic maps and the Schoen conjecture}.
\newblock {\em {J. Am. Math. Soc.}}, 30(3):799--817, 2017.

\bibitem{zbMATH03844576}
T.~{Milnor Klotz}.
\newblock {Harmonic maps and classical surface theory in Minkowski 3-space}.
\newblock {\em {Trans. Am. Math. Soc.}}, 280:161--185, 1983.

\bibitem{zbMATH00120199}
Y.~N. {Minsky}.
\newblock {Harmonic maps, length, and energy in Teichm\"uller space}.
\newblock {\em {J. Differ. Geom.}}, 35(1):151--217, 1992.

\bibitem{xin}
X.~Nie and A.~Seppi.
\newblock Regular domains and surfaces of constant {G}aussian curvature in
  three-dimensional affine space.
\newblock {\em To appear in Analysis \& PDE}, 2019.

\bibitem{zbMATH03608406}
J.~H. {Sampson}.
\newblock {Some properties and applications of harmonic mappings}.
\newblock {\em {Ann. Sci. \'Ec. Norm. Sup\'er. (4)}}, 11(2):211--228, 1978.

\bibitem{zbMATH00421703}
R.~M. {Schoen}.
\newblock {The role of harmonic mappings in rigidity and deformation problems}.
\newblock In {\em {Complex geometry. Proceedings of the Osaka international
  conference, held in Osaka, Japan, Dec. 13-18, 1990}}, pages 179--200. New
  York: Marcel Dekker, 1993.

\bibitem{zbMATH07060418}
A.~{Seppi}.
\newblock {Maximal surfaces in Anti-de Sitter space, width of convex hulls and
  quasiconformal extensions of quasisymmetric homeomorphisms}.
\newblock {\em {J. Eur. Math. Soc. (JEMS)}}, 21(6):1855--1913, 2019.

\bibitem{zbMATH07127319}
A.~{Seppi}.
\newblock {On the maximal dilatation of quasiconformal minimal Lagrangian
  extensions}.
\newblock {\em {Geom. Dedicata}}, 203:25--52, 2019.

\bibitem{graham}
G.~Smith.
\newblock On the {W}eyl problem in {M}inkowski space.
\newblock {\em Preprint, Arxiv 2005:00137}, 2020.

\bibitem{zbMATH00870993}
S.~{Trapani} and G.~{Valli}.
\newblock {One-harmonic maps on Riemann surfaces}.
\newblock {\em {Commun. Anal. Geom.}}, 3(4):645--681, 1995.

\bibitem{zbMATH00120216}
T.~Y.~H. {Wan}.
\newblock {Constant mean curvature surface, harmonic maps, and universal
  Teichm\"uller space}.
\newblock {\em {J. Differ. Geom.}}, 35(3):643--657, 1992.

\bibitem{zbMATH04069867}
M.~{Wolf}.
\newblock {The Teichm\"uller theory of harmonic maps}.
\newblock {\em {J. Differ. Geom.}}, 29(2):449--479, 1989.

\bibitem{zbMATH00034325}
M.~{Wolf}.
\newblock {High energy degeneration of harmonic maps between surfaces and rays
  in Teichm\"uller space}.
\newblock {\em {Topology}}, 30(4):517--540, 1991.

\bibitem{zbMATH00011240}
M.~{Wolf}.
\newblock {Infinite energy harmonic maps and degeneration of hyperbolic
  surfaces in moduli space}.
\newblock {\em {J. Differ. Geom.}}, 33(2):487--539, 1991.

\end{thebibliography}

\end{document}